\documentclass[11pt, a4paper, twoside,leqno]{amsart}
\usepackage[centering, totalwidth = 380pt, totalheight = 600pt]{geometry}
\usepackage{amssymb, amsmath, amsthm, listings, xspace}
\usepackage{enumerate, microtype, stmaryrd, url} 
\usepackage[latin1]{inputenc}
\usepackage[arrow, matrix, tips, curve, graph, rotate]{xy}
 \usepackage{color}
\definecolor{darkgreen}{rgb}{0,0.45,0}
\usepackage[colorlinks,citecolor=darkgreen,linkcolor=darkgreen]{hyperref}
\SelectTips{cm}{10}

\newcommand{\cat}[1]{\mathbf{#1}}
\newcommand{\op}{\mathrm{op}}

\newcommand{\thg}{{\mathord{\text{--}}}}

\renewcommand{\c}{,\,}

\DeclareMathAlphabet      {\mathbf}{OT1}{cmr}{b}{n}

\usepackage{tikz-cd}

\newcommand{\dcat}[1]{\cat{\mathbb #1}}
\newcommand{\cd}[2][]{\vcenter{\hbox{\xymatrix#1{#2}}}}

\makeatletter
\def\matrixobject@{%
   \edef \next@{={\DirectionfromtheDirection@ }}%
   \expandafter \toks@ \next@ \plainxy@
   \let\xy@@ix@=\xyq@@toksix@
   \xyFN@ \OBJECT@}
\let\xy@entry@@norm=\entry@@norm
\def\entry@@norm@patched{%
   \let\object@=\matrixobject@
   \xy@entry@@norm }
\AtBeginDocument{\let\entry@@norm\entry@@norm@patched}
\makeatother

\renewcommand{\phi}{\varphi}
\newcommand{\A}{{\mathcal A}}
\newcommand{\B}{{\mathcal B}}
\newcommand{\C}{{\mathcal C}}
\newcommand{\D}{{\mathcal D}}
\newcommand{\E}{{\mathcal E}}

\newcommand{\J}{{\mathcal J}}

\renewcommand{\L}{{\mathcal L}}
\newcommand{\M}{{\mathcal M}}

\newcommand{\R}{{\mathcal R}}

\newcommand{\xtor}[1]{\cdl[@1]{{} \ar[r]|-{\object@{|}}^{#1} & {}}}

\makeatletter

\newcommand{\setmanuallabel}[1]{\stepcounter{equation}{\edef\@currentlabel{\theequation}\label{#1}}}
\newcommand{\printmanuallabel}[1]{\stepcounter{equation}\text{(\theequation)}}

\def\hookleftarrowfill@{\arrowfill@\leftarrow\relbar{\relbar\joinrel\rhook}}
\def\twoheadleftarrowfill@{\arrowfill@\twoheadleftarrow\relbar\relbar}
\def\leftbararrowfill@{\arrowdoublefill@{\leftarrow\mkern-5mu}\relbar\mapstochar\relbar\relbar}
\def\Leftbararrowfill@{\arrowdoublefill@{\Leftarrow\mkern-2mu}\Relbar\Mapstochar\Relbar\Relbar}
\def\leftringarrowfill@{\arrowdoublefill@{\leftarrow\mkern-3mu}\relbar{\mkern-3mu\circ\mkern-2mu}\relbar\relbar}
\def\lefttriarrowfill@{\arrowfill@{\mathrel\triangleleft\mkern0.5mu\joinrel\relbar}\relbar\relbar}
\def\Lefttriarrowfill@{\arrowfill@{\mathrel\triangleleft\mkern1mu\joinrel\Relbar}\Relbar\Relbar}

\def\hookrightarrowfill@{\arrowfill@{\lhook\joinrel\relbar}\relbar\rightarrow}
\def\twoheadrightarrowfill@{\arrowfill@\relbar\relbar\twoheadrightarrow}
\def\rightbararrowfill@{\arrowdoublefill@{\relbar\mkern-0.5mu}\relbar\mapstochar\relbar\rightarrow}
\def\Rightbararrowfill@{\arrowdoublefill@{\Relbar\mkern-2mu}\Relbar\Mapstochar\Relbar\Rightarrow}
\def\rightringarrowfill@{\arrowdoublefill@\relbar\relbar{\mkern-2mu\circ\mkern-3mu}\relbar{\mkern-3mu\rightarrow}}
\def\righttriarrowfill@{\arrowfill@\relbar\relbar{\relbar\joinrel\mkern0.5mu\mathrel\triangleright}}
\def\Righttriarrowfill@{\arrowfill@\Relbar\Relbar{\Relbar\joinrel\mkern1mu\mathrel\triangleright}}

\def\leftrightarrowfill@{\arrowfill@\leftarrow\relbar\rightarrow}
\def\mapstofill@{\arrowfill@{\mapstochar\relbar}\relbar\rightarrow}

\renewcommand*\xleftarrow[2][]{\ext@arrow 20{20}0\leftarrowfill@{#1}{#2}}
\providecommand*\xLeftarrow[2][]{\ext@arrow 60{22}0{\Leftarrowfill@}{#1}{#2}}
\providecommand*\xhookleftarrow[2][]{\ext@arrow 10{20}0\hookleftarrowfill@{#1}{#2}}
\providecommand*\xtwoheadleftarrow[2][]{\ext@arrow 60{20}0\twoheadleftarrowfill@{#1}{#2}}
\providecommand*\xleftbararrow[2][]{\ext@arrow 10{22}0\leftbararrowfill@{#1}{#2}}
\providecommand*\xLeftbararrow[2][]{\ext@arrow 50{24}0\Leftbararrowfill@{#1}{#2}}
\providecommand*\xleftringarrow[2][]{\ext@arrow 10{26}0\leftringarrowfill@{#1}{#2}}
\providecommand*\xlefttriarrow[2][]{\ext@arrow 80{24}0\lefttriarrowfill@{#1}{#2}}
\providecommand*\xLefttriarrow[2][]{\ext@arrow 80{24}0\Lefttriarrowfill@{#1}{#2}}

\renewcommand*\xrightarrow[2][]{\ext@arrow 01{20}0\rightarrowfill@{#1}{#2}}
\providecommand*\xRightarrow[2][]{\ext@arrow 04{22}0{\Rightarrowfill@}{#1}{#2}}
\providecommand*\xhookrightarrow[2][]{\ext@arrow 00{20}0\hookrightarrowfill@{#1}{#2}}
\providecommand*\xtwoheadrightarrow[2][]{\ext@arrow 03{20}0\twoheadrightarrowfill@{#1}{#2}}
\providecommand*\xrightbararrow[2][]{\ext@arrow 01{22}0\rightbararrowfill@{#1}{#2}}
\providecommand*\xRightbararrow[2][]{\ext@arrow 04{24}0\Rightbararrowfill@{#1}{#2}}
\providecommand*\xrightringarrow[2][]{\ext@arrow 01{26}0\rightringarrowfill@{#1}{#2}}
\providecommand*\xrighttriarrow[2][]{\ext@arrow 07{24}0\righttriarrowfill@{#1}{#2}}
\providecommand*\xRighttriarrow[2][]{\ext@arrow 07{24}0\Righttriarrowfill@{#1}{#2}}

\providecommand*\xmapsto[2][]{\ext@arrow 01{20}0\mapstofill@{#1}{#2}}
\providecommand*\xleftrightarrow[2][]{\ext@arrow 10{22}0\leftrightarrowfill@{#1}{#2}}
\providecommand*\xLeftrightarrow[2][]{\ext@arrow 10{27}0{\Leftrightarrowfill@}{#1}{#2}}

\makeatother

\newcommand{\twocong}[2][0.5]{\ar@{}[#2] \save ?(#1)*{\cong}\restore}
\newcommand{\twoeq}[2][0.5]{\ar@{}[#2] \save ?(#1)*{=}\restore}
\newcommand{\rtwocell}[3][0.5]{\ar@{}[#2] \ar@{=>}?(#1)+/l 0.2cm/;?(#1)+/r 0.2cm/^{#3}}
\newcommand{\ltwocell}[3][0.5]{\ar@{}[#2] \ar@{=>}?(#1)+/r 0.2cm/;?(#1)+/l 0.2cm/^{#3}}
\newcommand{\ltwocello}[3][0.5]{\ar@{}[#2] \ar@{=>}?(#1)+/r 0.2cm/;?(#1)+/l 0.2cm/_{#3}}
\newcommand{\dtwocell}[3][0.5]{\ar@{}[#2] \ar@{=>}?(#1)+/u  0.2cm/;?(#1)+/d 0.2cm/^{#3}}
\newcommand{\dltwocell}[3][0.5]{\ar@{}[#2] \ar@{=>}?(#1)+/ur  0.2cm/;?(#1)+/dl 0.2cm/^{#3}}
\newcommand{\drtwocell}[3][0.5]{\ar@{}[#2] \ar@{=>}?(#1)+/ul  0.2cm/;?(#1)+/dr 0.2cm/^{#3}}
\newcommand{\dthreecell}[3][0.5]{\ar@{}[#2] \ar@3{->}?(#1)+/u  0.2cm/;?(#1)+/d 0.2cm/^{#3}}
\newcommand{\utwocell}[3][0.5]{\ar@{}[#2] \ar@{=>}?(#1)+/d 0.2cm/;?(#1)+/u 0.2cm/_{#3}}
\newcommand{\dtwocelltarg}[3][0.5]{\ar@{}#2 \ar@{=>}?(#1)+/u  0.2cm/;?(#1)+/d 0.2cm/^{#3}}
\newcommand{\utwocelltarg}[3][0.5]{\ar@{}#2 \ar@{=>}?(#1)+/d  0.2cm/;?(#1)+/u 0.2cm/_{#3}}

\newdir{(}{{}*!<0em,-.14em>-\cir<.14em>{l^r}}
\newdir{ (}{{}*!/-5pt/\dir{(}}
\newdir{ >}{{}*!/-5pt/\dir{>}}

\theoremstyle{definition}

\theoremstyle{plain}
\newtheorem{Thm}{Theorem}
\newtheorem{Prop}[Thm]{Proposition}

\newtheorem{Lemma}[Thm]{Lemma}
\numberwithin{equation}{section}

\theoremstyle{definition}
\newtheorem{Defn}[Thm]{Definition}

\newtheorem{Exs}[Thm]{Examples}
\newtheorem{Rk}[Thm]{Remark}

\renewcommand{\l}[1]{L{#1}}

\renewcommand{\r}[1]{R{#1}}

\newcommand{\m}[1]{\mu_{#1}}
\renewcommand{\c}[1]{\Delta_{#1}}

\newcommand{\lp}[1]{L'{#1}}

\newcommand{\rp}[1]{R'{#1}}

\newcommand{\Coalg}[1]{\mathsf{#1}\text-\cat{Coalg}}
\newcommand{\Alg}[1]{\mathsf{#1}\text-\cat{Alg}}

\newcommand{\DCoalg}[1]{\mathsf{#1}\text-\mathbb C\cat{oalg}}
\newcommand{\DAlg}[1]{\mathsf{#1}\text-\mathbb A\cat{lg}}
\newcommand{\Left}[1]{\cat{Left}\mathsf{#1}}
\newcommand{\Right}[1]{\cat{Right}\mathsf{#1}}
\newcommand{\Sem}{\cat{Sem}}
\newcommand{\Sq}[1]{\mathbb S\cat{q}(#1)}

\newcommand{\atwo}{{\mathbf 2}}

\newcommand{\alg}[1]{\boldsymbol{#1}}
\newcommand{\awfs}{awfs }

\newcommand{\Awfs}[1]{{\cat{AWFS}(#1)}}
\newcommand{\Oawfs}[1]{{\cat{LAWFS}(#1)}}
\newcommand{\Pawfs}[1]{{\cat{PAWFS}(#1)}}
\newcommand{\Cawfs}[1]{{\cat{LIFT}(#1)}}

\newcommand{\Oplax}{{\cat{AWFS}_\mathrm{oplax}}}
\newcommand{\LOplax}{{\cat{LAWFS}_\mathrm{oplax}}}

\newcommand{\Cat}{{\cat{Cat}}}
\newcommand{\DBL}{{\cat{DBL}}}
\newcommand{\LLP}{\cat{LLP}}
\newcommand{\RLP}{\cat{RLP}}

\newcommand{\CAT}{{\cat{CAT}}}
\newcommand{\Set}{{\cat{Set}}}
\newcommand{\SET}{{\cat{SET}}}

\def\black{\color{black}}

\begin{document}

\leftmargini=2em \title{An orthogonal approach to algebraic weak factorisation systems}

\author{John Bourke}
\address{Department of Mathematics and Statistics, Faculty of Science, Masaryk University, Brno, Czech Republic}
\email{bourkej@math.muni.cz} 
\thanks{The author acknowledges the support of the Grant agency of the Czech Republic under the grant 22-02964S}

\begin{abstract}
We describe an equivalent formulation of algebraic weak factorisation systems, not involving monads and comonads, but involving double categories of morphisms equipped with a lifting operation satisfying lifting and factorisation axioms.
\end{abstract} 
\date\today
\maketitle

\section{Introduction}

The concept of a factorisation system abstracts the relationship between surjections and injections in algebraic categories.  Originating in the early days of category theory \cite{Duality}, the concept was slowly refined \cite{Isbell57,Kennison68, Ringel70} and the importance of orthogonality gradually came to light: two classes of morphisms $\E$ and $\M$ in a category $\C$ are said to be orthogonal, written $\E \perp \M$, if each commutative square 
\begin{equation}\label{eq:lifting}
\cd{A \ar[d]_{e \in \E} \ar[r]^-{u} &  C \ar[d]^{m \in \M}  \\
B \ar@{.>}[ur]^{\exists !} \ar[r]_-{v}  &D}
\end{equation}
has a unique diagonal filler.  Writing ${\E}^{\perp} = \{f \in \C : \E \perp f \}$ and ${}^{\perp}\M = \{f \in \C: f \perp \M \}$, the orthogonality of $\E$ and $\M$ can be written as $\E \subseteq {}^{\perp}\M$ or, equivalently, $\M  \subseteq {\E}^{\perp}$.  In one of its well known modern formulations \cite{Freyd72}, a factorisation system on $\C$ consists of two classes of morphisms $\E$ and $\M$ such that:
\begin{enumerate}
\item \emph{Axiom of orthogonality:} $\E = {}^{\perp}\M$ and $\M  = {\E}^{\perp}$;
\item \emph{Axiom of factorisation:} each morphism $f \in \C$ is of the form $f = m \circ e$ where $m \in \M$ and $e \in \E$.
\end{enumerate}

Weak factorisation systems were first considered by Quillen \cite{Quillen1967} in his definition of a closed model category, which involves two interconnected weak factorisation systems.  The concept is identical in form to the above one, with the exception that orthogonality is replaced by weak orthogonality (a.k.a. the lifting property): in the commutative square \eqref{eq:lifting}, we require that there exists a diagonal filler as before, but it is not required to be unique.  
This makes the concept of weak factorisation system more expressive, allowing it to capture homotopical structures such as Kan fibrations and Kan complexes which involve liftings that are not unique, but only up to a suitable notion of homotopy.\begin{footnote}{As understood by Quillen from the outset, under mild completeness conditions such a notion of homotopy is intrinsically encoded within the weak factorisation system itself.}\end{footnote}

Dropping uniqueness in orthogonality has, however, some drawbacks.  One loses, for example, the functoriality of factorisations.  This is at odds with the vast majority of examples of weak factorisation systems, which can be made functorial by Quillen's small object argument.  Such issues led category theorists to investigate concepts intermediate between factorisation systems and weak factorisation systems --- see, for instance \cite{Rosicky2002} --- and ultimately led to Grandis and Tholen's natural weak factorisation systems \cite{Grandis2006}, refined by Garner \cite{Garner2011}, and nowadays known as \emph{algebraic weak factorisation systems} (awfs).

The definition of an awfs is rather different in form to those discussed above: it involves a so-called functorial factorisation together with an interacting comonad and monad pair $(\mathsf L,\mathsf R)$ on the arrow category $\C^{\atwo}$ of $\C$ such that a certain natural transformation is a distributive law.  The categories $\Coalg{L}$ and $\Alg{R}$ are then thought of as consisting of the left and right maps of the awfs --- being a left map or right map is no longer a property of a morphism but additional structure.  
In this context, given an $\mathsf L$-coalgebra $\alg f = (f,s)$, an $\mathsf
R$-algebra $\alg g = (g,p)$ and a commuting square
\begin{equation}\label{eq:chosen}
\cd{
A \ar[d]_{f} \ar[rr]^{u}  && C \ar[d]^{g} \\
B \ar@{.>}[urr]|{\Phi_{\alg f, \alg g}(u,v)} \ar[rr]_{v}  && D} \qquad \qquad
\end{equation}
there is a \emph{chosen} diagonal filler, as depicted, and these chosen fillers satisfy a number of compatibilities.  Thus awfs encode something in between orthogonality and weak orthogonality --- they encode a structured form of the lifting property involving chosen diagonal fillers satisfying further compatibilities.  This allows them to capture classical homotopical structures \cite{Garner2011, Riehl2011Algebraic}, as well as structures borne by morphisms such as Grothendieck fibrations \cite{Bourke2016Accessible} and various flavours of structured fibration arising in homotopy type theory \cite{Gambino2017, Berg2020} that lie beyond the scope of factorisation systems and weak factorisation systems.

The goal of this paper is to give a new definition of awfs --- not mentioning functorial factorisations, comonads or monads --- in the same spirit as the definitions of factorisation system and weak factorisation system described above.  As such, the key player in the definition will be a structured form of the lifting property --- namely, the concept of \emph{lifting operation} defined in \cite{Bourke2016Accessible} --- and accordingly we call the concept a \emph{lifting awfs}.  Of course whilst the definition is new, the underlying concept is not --- the main theorem of the paper establishes an equivalence between the categories of awfs on $\C$ and of lifting awfs on $\C$.  One point of view is that an awfs, as classically presented, is a form of theory whilst the associated lifting awfs represents the semantics of the theory.

Before beginning the paper, we will now briefly describe the notion of lifting awfs in detail.  As mentioned above, the chosen fillers in \eqref{eq:chosen} satisfy a number of compatibilities.  These break down into two kinds, horizontal and vertical.  The horizontal compatibilities say that the chosen fillers are natural in morphisms of $\mathsf L$-coalgebras (on the left) and of $\mathsf R$-algebras (on the right).   The vertical compatibilities refer to a vertical composition of $\mathsf L$-coalgebras and of $\mathsf R$-algebras, which extends the categories $\Coalg{L}$ and $\Alg{L}$ over $\C^{\atwo}$ to \emph{double categories} $\DCoalg{L}$ and $\DAlg{R}$ over $\Sq{\C}$, the double category of commutative squares in $\C$; the chosen fillers $\phi$ now form a double-categorical lifting operation of $\DCoalg{L}$ over $\DAlg{R}$ in the sense of \cite{Bourke2016Accessible}.  

Accordingly, we define a lifting awfs on $\C$ to consist of double categories $U \colon  \dcat L \to \Sq{\C}$ and $V \colon  \dcat R \to \Sq{\C}$ over $\Sq{\C}$ together with an $(\dcat L, \dcat R)$-lifting operation $\phi$ such that:
\begin{enumerate}
\item \emph{Axiom of lifting:} the induced $\phi_l \colon  \dcat L \to \LLP(\dcat R)$ and $\phi_r \colon  \dcat R \to \RLP(\dcat L)$ are invertible;
\item \emph{Axiom of factorisation:} each morphism $f \in \C$ admits a bi-universal factorisation $f = V_1 h \circ U_1 g$
where $g$ is a vertical morphism of $\dcat L$ and $h$ is a vertical morphism of $\dcat R$.
\end{enumerate}

Some of the terminology in the above definition has not been fully defined yet, but the reader should observe its similarity to the definitions of factorisation system and weak factorisation system described above.  In Section~\ref{sect:liftingawfs}, we give a self-contained definition of lifting awfs, describe a few examples and develop some of their theory.  By putting the new definition at the start of the paper, without reference to the classical definition of awfs, our aim is to present it as a primitive concept.  In Section~\ref{sect:awfs}, we recall the necessary facts about awfs, in particular the key results about their double categorical semantics developed in \cite{Riehl2011Algebraic,Bourke2016Accessible}.  Using these results heavily, in Section~\ref{sect:equivalence} we prove the equivalence between awfs and lifting awfs. 

\section{Lifting awfs}\label{sect:liftingawfs}

In this section, we begin with a brief introduction to double categories, motivating the concept through the lens of factorisation systems.  We then introduce \cite{Bourke2016Accessible}'s notion of a double-categorical lifting operation, which built on earlier work of Garner \cite{Garner2011} by introducing vertical compatibilities.  Using this concept, we give the definition of lifting awfs.

Our foundational assumptions follow those of \cite{Bourke2016Accessible} --- in particular, 
$\cat{Set}$ and $\cat{SET}$ are the categories of small and large
sets; $\cat{Cat}$ and $\cat{CAT}$ are the categories of small
categories and of locally small categories.  We will treat $\cat{Cat}$ and $\cat{CAT}$ as categories, not as $2$-categories.
Throughout the paper, all categories will be assumed to be locally small  except for ones
whose names, like $\SET$ or $\CAT$, are in capital letters.

\subsection{Double categories of maps}

In the context of classical orthogonal factorisation systems, we have two classes of morphisms $\E$ and $\M$ of a category $\C$, closed under composition and containing the identities.  Let us focus on $\E$.  There are two distinct ways of viewing $\E$ as a category:
\begin{itemize}
\item as the lluf subcategory of $\C$ containing all objects but just those morphisms belonging to $\E$;
\item as the full subcategory of the arrow category $\C^{\atwo}$ whose objects are the morphisms of $\E$.  
\end{itemize}
Both perspectives are important, and it is natural to ask for a single structure encoding both --- this is achieved with the concept of a double category, first introduced in \cite{Ehresmann}.  A double category $\dcat{J}$ can be succinctly defined as an internal category

$$\xy
(0,0)*+{\J_1 \times_{\J_0} \J_1}="c0"; (20,0)*+{\J_{1}}="b0";(40,0)*+{\J_{0}}="a0";
{\ar@<1.5ex>^{d} "b0"; "a0"}; 
{\ar@<0ex>|{i} "a0"; "b0"}; 
{\ar@<-1.5ex>_{c} "b0"; "a0"}; 
{\ar@<0ex>^-{m} "c0"; "b0"}; 
\endxy
$$ in $\CAT$.  In elementary terms, this involves:
\begin{itemize} 
\item \emph{objects} (the objects of $\J_0$);
\item \emph{horizontal arrows} (the arrows of $\J_0$);
\item \emph{vertical arrows} (the objects of $\J_1$, with domain and codomain assigned by $d$ and $c$);
\item \emph{squares} 
$$\xy
(0,0)*+{A}="b0"; (15,0)*+{C}="c0"; (0,-15)*+{B}="d0"; (15,-15)*+{D}="e0";
{\ar ^{r} "b0";"c0"};
{\ar _{f} "b0";"d0"};
{\ar ^{g}"c0";"e0"};
{\ar_{s} "d0";"e0"};
{\ar@{=>}^{\alpha}(5,-7)*+{};(10,-7)*+{}};
\endxy$$
where $f$ and $g$ are vertical morphisms and $r,s$ horizontal.  (In terms of the internal category, $\alpha:f \to g$ is an arrow of $\J_1$ with horizontal domain $r$ and codomain $s$ assigned by $d$ and $c$.)
\end{itemize}
Objects and horizontal arrows form a category ($\J_0$) as do objects and vertical arrows (defined using the internal category maps $m$ and $i$).  Squares can be composed both horizontally (as in $\J_1$) and vertically (using $m$) and both compositions of squares are associative and unital.  Moreover they satisfy an interchange law involving four squares, which amounts to functoriality of $m$.

A double functor $F \colon  \dcat{J} \to \dcat{K}$ is simply an internal functor between internal categories; it preserves all of the structure of the double category described in elementary terms above.  We write $\DBL$ for the category of double categories and double functors.

There are many interesting examples of double categories, such as the double category of sets, functions and relations.  If we allow ourselves to work with \emph{pseudo-double categories} rather than the \emph{strict} double categories considered here, we obtain access to examples involving spans, bimodules and profunctors.  For further information on double categories, we refer the reader to \cite{Shulman,GrandisPare}.

Let us return towards our guiding examples, which are rather simple kinds of double categories, in which each square is determined by its four bounding edges. 

\begin{Exs}\label{examples:maps1}
Given a category $\C$ one can form the double category $\Sq{\C}$ whose objects are those of $\C$, whose vertical and horizontal morphisms coincide as the morphisms of $\C$, and whose squares are simply the commutative squares in $\C$.

Now let $\E$ be a class of morphisms in $\C$ closed under composition and containing the identities.  Then we can form a double category $\dcat{D}(\E)$ whose objects and horizontal arrows are those of $\C$, whose vertical arrows are those of $\E$ and whose squares are again just commutative squares as depicted below.
$$\xy
(0,0)*+{A}="b0"; (15,0)*+{C}="c0"; (0,-15)*+{B}="d0"; (15,-15)*+{D}="e0";
{\ar ^{r} "b0";"c0"};
{\ar _{e \in \E} "b0";"d0"};
{\ar ^{e' \in \E}"c0";"e0"};
{\ar_{s} "d0";"e0"};
\endxy$$
This comes equipped with a forgetful double functor $U \colon  \dcat{D}(\E) \to \Sq{\C}$ which is the identity on objects and horizontal morphisms.  The reader will observe that this double category encodes both of the categories associated to $\E$ mentioned at the beginning of this section.
\end{Exs}

Central to our concerns will be double functors of the form $U \colon  \dcat{J} \to \Sq{\C}$.  Written internally to $\CAT$, such a double functor gives rise to a diagram

\begin{equation*}
\xymatrix{
  \J_1 \ar@<-5pt>[d]_d\ar@<5pt>[d]^c \ar[r]^{U_1} &
  \C^\atwo
  \ar@<-5pt>[d]_d\ar@<5pt>[d]^c \\
  \ar[u]|i \J_0 \ar[r]_{U_0} & \C\rlap{ .} \ar[u]|i }
\end{equation*}

where $\J_0$ is the category of objects and horizontal morphisms, and $\J_1$ the category of vertical morphisms and squares.  It is often the case that $U_0$ is the identity functor and $U_1$ is faithful, in which case, following \cite{Bourke2016Accessible}, we call it a \emph{concrete} double category over $\C$.  

Of course $U \colon  \dcat{D}(\E) \to \Sq{\C}$ of Examples~\ref{examples:maps1} is a concrete double category over $\Sq{\C}$.  In fact, in this example $U_1$ is \emph{fully faithful} --- this corresponds to the fact that the vertical morphisms of $\dcat{D}(\E)$ are simply morphisms of $\C$ satisfying a \emph{property}.  Relaxing fully faithfulness to faithfulness allows us to capture morphisms of $\C$ \emph{equipped with structure}.  

\black

\begin{Exs}\label{examples:maps2}
There are many interesting structured examples over $\Cat$, arising from the different flavours of adjunction and fibration, and we will focus on two of these.

By a \emph{split reflection}, we mean a functor $u\colon  A \to B$ together with an adjunction $f \dashv u$ with identity counit --- in particular $fu=1$.  These are the vertical morphisms of a concrete double category $\dcat{SplRef}(\Cat) \to \Sq{\Cat}$ whose squares are commutative squares commuting with the adjunctions strictly. Composition of adjunctions (restricted to split reflections) provides the vertical composition of the double category.

A \emph{cloven fibration} $(f,\theta)$ is a Grothendieck fibration $u\colon  A \to B$ together with a cleavage --- that is, for each $f\colon b \to ua$, a choice of $u$-cartesian lifting $\theta_{a,f}\colon a^{*}f \to a$.  It is, moreover, a \emph{split fibration} if $\theta$ is compatible with identities and composition in $B$.  Split fibrations form the vertical morphisms of a concrete double category $\dcat{SplFib}(\Cat) \to \Sq{\Cat}$ whose squares are commutative squares preserving the chosen cartesian liftings.  Composition of split fibrations provides the vertical composition of the double category.
\end{Exs}

\subsection{Lifting operations for double categories of maps}\label{sect:lifting}

Consider again two classes of morphisms $\E, \M$ of $\C$ closed under composition and identities, and consider the associated concrete double categories $U\colon \dcat{D}(\E) \to \Sq{\C}$ and $V\colon \dcat{D}(\M) \to \Sq{\C}$ over $\C$.  In these terms, orthogonality of $\E$ and $\M$ says that given vertical morphisms $e\colon A \to B \in \dcat{D}(\E)$, $m\colon X \to Y \in \dcat{D}(\M)$ and a commutative square
\begin{equation*}
\cd{ UA \ar[d]_{Ue} \ar[r]^-{u} &  VX \ar[d]^{Vm}  \\
UB \ar@{.>}[ur]|{\exists !} \ar[r]^-{v}  &VY}
\end{equation*}
there exists a unique diagonal filler, as depicted.  

Moving to more structured examples, such as the split reflections and split fibrations of Example~\ref{examples:maps2}, the diagonal filler will no longer be unique.  Nonetheless, there will be a canonical choice of filler and these canonical fillers will satisfy a number of compatibilities (which are automatic in the orthogonal case).  This structured system of fillers is encapsulated by the following notion of lifting operation, first defined in \cite{Bourke2016Accessible}.\black

Given $U \colon  \dcat{L} \to \Sq{\C}$ and $V \colon  \dcat{R} \to
\Sq{\C}$, an \emph{$(\dcat{L},\dcat{R})$-lifting operation} is a family of functions $\phi_{j,k}$, indexed by vertical arrows $j \colon  A \to B \in \dcat{L}$
and vertical arrows $k \colon  X \to Y \in \dcat{R}$, which assign to each commuting square
\begin{equation*}
\cd{ UA \ar[d]_{Uj} \ar[r]^-{u} &  VX \ar[d]^{Vk}  \\
UB \ar@{.>}[ur]|{\phi_{j,k}(u,v)} \ar[r]^-{v}  &VY}
\end{equation*}
a diagonal filler, as indicated, making both triangles commute.  These diagonal fillers are required to satisfy:
\begin{itemize}
\item \emph{Horizontal compatibilities:} the diagonal fillers are required to be natural in squares of $\dcat{L}$ and $\dcat{R}$ --- for $\dcat{L}$, this says that given a morphism $r\colon i \to j \in \L_1$, we have the equality of diagonals in \[
\cd{UA \ar[d]_{Ui} \ar[r]^-{Ur_0}  & UC \ar[d]_{Uj} \ar[r]^s & VX  \ar[d]^{Vk} \\ 
UB \ar[r]_-{Ur_1}  & UD \ar@{.>}[ur]|{\phi_{j,k}} \ar[r]_t & VY} 
  \quad = \quad 
  \cd{UA \ar[d]_{Ui} \ar[rr]^-{s.Ur_0}  && VX \ar[d]^{Vk} \\ 
  UB \ar@{.>}[urr]|{\phi_{i,k}} \ar[rr]_-{t.Ur_1}  && VY}
\]
where we have omitted certain labels for $\phi$.  Naturality in squares of $\dcat{R}$ gives the corresponding condition on the right.
\item \emph{Vertical compatibilities:} the diagonal fillers are required to respect composition of vertical morphisms in $\dcat{L}$ and $\dcat{R}$ --- given such a composable pair $j \circ i\colon A \to B \to C \in \dcat{L}$, this says that we have the equality of diagonals  from bottom left to top right in
\begin{equation*}
\cd[@-0.6em@C+1.2em]{
  UA \ar[d]_{Ui} \ar[rr]^s && VX \ar[dd]^{Vk} \\
  UB \ar[d]_{Uj}  \ar@{.>}[urr]^{{\phi_{i,k}}}\\
  UC \ar[rr]_-{t}  \ar@{.>}[uurr]_{{\phi_{j,k}}} &&  VY }
  \quad = \quad
  \cd[@-0.6em@C+1.2em]{
  UA \ar[d]_{Ui} \ar[rr]^s && VX \ar[dd]^{Vk} \\
  UB \ar[d]_{Uj} \\
  UC \ar[rr]_-{t}  \ar@{.>}[uurr]_{{\phi_{j \circ i,k}}} &&  VY }
\end{equation*}
whilst respecting vertical composition in $\dcat{R}$ concerns the corresponding  condition but with the composite on the right.
\end{itemize}

The assignation sending $\dcat{L}$ and $\dcat{R}$ to the collection of $(\dcat{L}, \dcat{R})$-lifting operations underlies a functor $$\dcat{Lift} \colon  (\DBL / \Sq{\C})^\op \times (\DBL / \Sq{\C})^\op \to \SET \hspace{0.1cm} .$$
 Given $F\colon \dcat{L} \to \dcat{L'}$ and $G\colon \dcat{R} \to \dcat{R'}$ in $\DBL / \Sq{\C}$, this sends an $(\dcat{L'},\dcat{R'})$-lifting operation $\phi$ to the $(\dcat{L},\dcat{R})$-lifting operation $\phi_{F,G}$ with $(\phi_{F,G})_{j,k} = \phi_{Fj,Gk}$ for $j$ a vertical morphism of $\dcat{L}$ and $k$ a vertical morphism of $\dcat{R}$.

 The functors $\dcat{Lift}(\dcat{L}, \thg)$ and $\dcat{Lift}(\thg, \dcat{R})$ are represented by concrete double categories $\LLP(\dcat{L}) \to \Sq{\C}$ and $\RLP(\dcat{R}) \to \Sq{\C}$ respectively.   A vertical morphism of $\LLP(\dcat{L})$ consists of a pair $(f,\phi)$ where $f$ is a morphism of $\C$ and $\phi$ is an $(\dcat{L},f)$-lifting operation\begin{footnote}{Here $f$ is considered as a double-functor $f\colon \atwo_v \to \Sq{\C}$ from the free double category containing a vertical arrow.}\end{footnote} as depicted below left
\begin{equation*}
\cd{ UA \ar[d]_{Uj} \ar[r]^-{u} &  X \ar[d]^{f}  \\
UB \ar@{.>}[ur]|{\phi_{j}(u,v)} \ar[r]_-{v}  & Y}
\hspace{2cm}
\cd{ X \ar[d]_{f} \ar[r]^-{u} &  VC \ar[d]^{Vk}  \\
Y \ar@{.>}[ur]|{\phi_{k}(u,v)} \ar[r]_-{v}  &VD}
\end{equation*} 
whilst a vertical morphism of $\RLP(\dcat{R})$ consists of a pair $(f,\phi)$ where $f$ is a morphism of $\C$ and $\phi$ is an $(f,\dcat{R})$-lifting operation, as depicted above right.  In either case, the squares $(f,\phi) \to (g,\theta)$ are the commutative squares $f \to g$ commuting with the lifting operations $\phi$ and $\theta$.

The representations are described by natural bijections
$$\dcat{Lift}(\dcat{L}, \dcat{R}) \cong \DBL/\Sq{\C}(\dcat{L},\LLP(\dcat{R})): \phi \mapsto \phi_l$$
and 
$$\dcat{Lift}(\dcat{L}, \dcat{R}) \cong \DBL/\Sq{\C}(\dcat{R},\RLP(\dcat{L})): \phi \mapsto \phi_r$$
and the composite bijection
$$\DBL/\Sq{\C}(\dcat{L},\LLP(\dcat{R})) \cong \dcat{Lift}(\dcat{L}, \dcat{R}) \cong \DBL/\Sq{\C}(\dcat{R},\RLP(\dcat{L}))$$
induces an adjunction
\begin{equation*} 
\cd[@C+1em]{(\DBL/\Sq{\C})^\op
  \ar@{}[r]|-{\bot} \ar@<-4.5pt>[r]_-{{\RLP}} & \DBL/\Sq{\C}\rlap{ .}
  \ar@<-4.5pt>[l]_-{\LLP}}\label{eq:34}
\end{equation*}
modifying the usual Galois connection between classes of morphisms in a category $\C$.  At $U\colon \dcat{J} \to \Sq{\C}$, the unit and counit are given by morphisms $\eta_{\dcat{J}}\colon \dcat{J} \to \RLP(\LLP(\dcat{J}))$ and $\epsilon_{\dcat{J}}\colon \dcat{J} \to \LLP(\RLP(\dcat{J}))$.

\subsection{Lifting awfs}

We now turn to the definition of \emph{lifting awfs}.  The basic context for a lifting awfs consists of a pair $U \colon  \dcat L \rightarrow \Sq{\C}$ and $V \colon  \dcat R \to \Sq{\C}$ together with an $(\dcat L, \dcat R)$-lifting operation $\phi$.   Henceforth we refer to such a triple $(\dcat L,\phi,\dcat R)$ as a \emph{lifting structure} on $\C$, noting that we abuse notation by leaving the morphisms $U$ and $V$ implicit.

Given a lifting structure $(\dcat L,\phi,\dcat R)$, we can ask whether the following two axioms hold.

\begin{enumerate}
\item \emph{Axiom of lifting:} the induced maps $\phi_l\colon  \dcat L \to \LLP(\dcat R)$ and $\phi_r \colon  \dcat R \to \RLP(\dcat L)$ are invertible;
\item \emph{Axiom of factorisation:} each morphism $f\colon A \to B \in \C$ admits a factorisation 
\begin{equation*}
\xymatrix{
A \ar[r]^{U_{1}g} & C \ar[r]^{V_{1}h} & B}
\end{equation*}
 which is bi-universal, in the sense that $(1_A,V_1 f)\colon U_1 g \to f$ is $U_1$-couniversal and $(U_{1}g,1_B)\colon f \to V_{1}h$ is $V_1$-universal.
\end{enumerate}

\begin{Defn}
A lifting structure $(\dcat L,\phi,\dcat R)$  is
\begin{itemize}
\item a pre-awfs if it satisfies the axiom of lifting;
\item a lifting awfs if it satisfies the axioms of lifting and factorisation.
\end{itemize}
\end{Defn}

The definition of a pre-awfs is motivated by the notion of a pre-factorisation system \cite{Freyd72}, which consists of two classes of maps $\E$ and $\M$ satisfying $\E = {}^{\perp}\M$ and $\M = \E^{\perp}$.  We remark that the usual definition of awfs, presented in Section~\ref{sec:algebr-weak-fact}, does not easily lead to a definition of pre-awfs, since it has the factorisations built in as primitive structure.

Lifting awfs are the central notion of this paper, and we will prove their equivalence with awfs in Theorem~\ref{thm:equivalence}.  We note that, in fact, there is some redundancy in the definition of a lifting awfs. Indeed it suffices to ask that the factorisation has \emph{either} universal property --- the other one then following automatically.  See Proposition~\ref{prop:redundant}.

\begin{Exs}\label{ex:second}
Our first goal is to convince the reader that the notion of lifting awfs is natural as a standalone concept, and to this end we now give a couple of guiding examples.  Both were considered as awfs in the paper that introduced the concept \cite{Grandis2006}.  The second example, though simple enough, displays all of the structure of a lifting awfs in a non-trivial way.

\begin{enumerate}[(i)]
\item Let $(\E,\M)$ be a pre-factorisation system on $\C$ and $\dcat{D}(\E)$ and $\dcat{D}(\M)$ be the corresponding concrete double categories of Example~\ref{examples:maps1}.  By orthogonality of $\E$ and $\M$, there exists a unique $(\dcat{D}(\E),\dcat{D}(\M))$-lifting operation $\phi$ and we will show that this makes $(\dcat{D}(\E),\phi,\dcat{D}(\M))$ a pre-awfs.  

To this end, let us first show that the induced map $\phi_r \colon \dcat{D}(\M) = \dcat{D}(\E^{\perp}) \to \RLP(\dcat{D}(\E))$ equipping a morphism with the unique choice of liftings is invertible.  So suppose that $(f,\phi) \in \RLP(\dcat{D}(\E))$.  Then we obtain a diagonal filler in the square below left and we claim it is the unique such filler.
\begin{equation*}
\xymatrix{
A \ar[d]_{e \in \E} \ar[r]^{r} & C \ar[d]^{f} \\
B \ar[ur]|{\phi_e(r,s)}  \ar[r]_{s} & D}
\hspace{1cm}
\xymatrix{
A \ar[d]_{e \in \E} \ar[r]^{e} & B \ar[d]_{1}  \ar[r]^{t} & C \ar[d]^{f} \\
B   \ar[r]_{1} & B \ar[ur]|{\phi_{1}(t,s)}  \ar[r]_{s} & D}
\end{equation*}
Indeed, a second lifting $t$ would induce a diagram as above right.  Now certainly $\phi_{1}(t,s) = t$ since the upper triangle commutes; naturality of $\phi$ in the left square of that diagram hence forces $\phi_e(r,s)=t$, as required.  Hence $\phi_r$ is bijective on vertical morphisms and fully faithfulness on squares follows from the fact that since the diagonal fillers are unique, they are preserved by any commutative square.  A dual argument shows that $\phi_l$ is invertible so that we have verified the axiom of lifting.

If $(\E,\M)$ is a factorisation system on $\C$ then orthogonality of the two classes ensures that the $(\E,\M)$-factorisation $f=m \circ e$ of a morphism gives a coreflection $(1,m)\colon e \to f$ into $\dcat{D}(\E)_1$ and reflection $(e,1)\colon f \to m$ into $\dcat{D}(\M)_1$, establishing the axiom of factorisation.  (That factorisation systems give rise to such (co)reflections in the arrow category was observed already in the 1960's \cite{Ehrbar}.  See also \cite{MacDonald} and \cite{Im}.\black)

\label{item:prefact}  

\item Continuing Example~\ref{examples:maps2}, we now describe in detail the lifting awfs on $\Cat$ whose left and right maps are the split reflections and split fibrations.  This is a close relative of the comprehensive factorisation system \cite{StreetComprehensive}, whose right class consists of the \emph{discrete} fibrations.

To get started, consider a commutative square in $\Cat$
\begin{equation*}
\cd{A \ar[d]_{u} \ar[r]^-{r} &  C \ar[d]^{g}  \\
B \ar@{.>}[ur]|{k} \ar[r]^-{s}  &D}
\end{equation*}
in which $f \dashv u$ is a split reflection and $(g,\theta)$ a split fibration.  We will describe a canonical diagonal filler $k$.  

At an object $b \in B$, applying $s$ to the unit component gives $s\eta_b\colon sb \to sufb=grfb$.  We denote its chosen $g$-cartesian lifting $\theta(s\eta_b,rfb)$ by $\theta_b := \theta(s\eta_b,rfb)\colon kb \to rfb$.  As a lifting, we have that $gkb=sb$ as required.  Since at $a \in A$ we have $\eta_{ua}=1$ and since $(g,\theta)$ is split, the cartesian lifting of the identity $s\eta_{ua}\colon sua=gra$ is the identity on $ra$, so that $kua = ra$.  In particular, both triangles commute at the level of objects.  

Now consider $\alpha\colon b \to b' \in B$.  Using the universal property of the cartesian lifting $\theta_{b'}$, we define $k\alpha$ as the unique morphism making the square
\begin{equation*}
\cd{kb \ar[d]_{\theta_b} \ar[r]^-{k\alpha} &  kb' \ar[d]^{\theta_{b'}}  \\
rfb \ar[r]^-{rf\alpha}  &rfb'}
\end{equation*}
commute and satisfying $gk\alpha = s\alpha$.  For $\alpha = u \beta$, the above uniqueness property implies that $r\beta = ku\beta$ so that both triangles commute.  Uniqueness also implies functoriality of $k$.

Accordingly, we define our $(\dcat{SplRef},\dcat{SplFib})$-lifting operation $\phi$ by $\phi(f \dashv u,(g,\theta) ,r,s)=k$.  We leave it as an exercise to the reader to check the compatibility axioms for a double-categorical lifting operation, noting that the vertical compatibilities make full use of the fact that the given reflections and fibrations are split.

A functor $f\colon A \to B$ has a bi-universal factorisation through the comma category $B/f$ as below.

\begin{equation*}
\xymatrix{
A \ar[r]^-{i_f} & B/f \ar[r]^-{d_f} & B}
\end{equation*}
On objects, the functor $i_f$ sends $a$ to $(1\colon fa \to fa,a)$ whilst $d_f(\alpha\colon b \to fa,a) = b$.  We have a split reflection $c_f \dashv i_f$ where $c_f$ is given on objects by $c_f(\alpha\colon b \to fa,a)=a$ and with unit $\eta\colon 1 \to c_fi_f$ as below left.
\begin{equation*}
\xymatrix{
b \ar[d]_{\alpha} \ar[r]^{\alpha} & fa \ar[d]^{1} \\
fa \ar[r]_{1} & fa
}
\hspace{2cm}
\xymatrix{
c \ar[d]_{\beta} \ar[r]^{\alpha  \beta} & fa \ar[d]^{1} \\
b \ar[r]^{\alpha} & fa
}
\end{equation*}
The functor $d_f$ is a split fibration: given $\beta\colon c \to b = d(\alpha\colon b \to fa,a)$ its chosen cartesian lift is depicted on the right above.  Street \cite{StreetYoneda} observed that with this choice of lifting $d_f\colon B/f \to B$ is indeed the free split fibration: the main point is that each object $(\alpha\colon b \to fa,a)$ of $B/f$ arises from a chosen cartesian lifting (see the diagram above left).  The universal property of the comma category $B/f$ as a limit, allows one to easily construct functors into it, and so verify that $i_f\colon A \to B/f$ is the cofree split reflection.  

To verify the axiom of lifting, we must show that $\phi_l\colon  \dcat{SplRef} \to \LLP(\dcat{SplFib})$  and $\phi_r \colon  \dcat{SplFib} \to \RLP(\dcat{SplRef})$ are invertible --- here, we will only outline the inverses on objects.  Given $(f,\theta) \in \RLP(\dcat{SplRef})$ we must equip $f$ with split fibration structure.  The lifting function $\theta$ provides a chosen lifting as below left --- it ultimately assigns to $(\alpha \colon b \to fa,a)$ the domain of its cartesian lifting.  
\begin{equation*}
\xymatrix{
A \ar[d]_{i_f} \ar[r]^{1} & A \ar[d]^{f} \\
B/f \ar[ur]|{\theta} \ar[r]_{d_f} & B
}
\hspace{2cm}
\xymatrix{
A \ar[d]_{f} \ar[r]^{i_f} & B/f \ar[d]^{d_f} \\
B \ar[ur]|{\theta} \ar[r]_{1} & B
}
\end{equation*}
In fact, the assignment $f \mapsto d_f$ is, as observed by Street \cite{StreetYoneda}, a monad on $\Cat/B$ whose algebras are precisely the split fibrations --- so far we have equipped $f$ with the structure of an algebra for the pointed endofunctor and the vertical compatibilities for $\theta$ allow one to establish compatibility with the multiplication, and so the split fibration structure.

For the inverse on objects to $\phi_l\colon  \dcat{SplRef} \to \LLP(\dcat{SplFib})$ consider $(f,\theta) \in \LLP(\dcat{SplFib})$.  The lifting above right induces a functor $g\colon B \to A$ and natural transformation $\eta\colon 1 \to fg$ such that $\eta f$ is the identity; the other triangle equation for a split reflection can be obtained from the vertical compatibilities for $\theta$; this time it amounts to considering the comonad structure for $f \mapsto i_f$.
\end{enumerate}
\end{Exs}

 For a discussion of homotopical examples, see Examples~\ref{exs:homotopy}.\black

\subsection{The category of lifting awfs on $\C$}

In order to understand the category of lifting awfs on $\C$ cleanly, it is natural to define it as a full subcategory of the category $\Cawfs{\C}$ of \emph{lifting structures} on $\C$.  

This larger category admits a simple definition: indeed, lifting structures $(\dcat L,\phi,\dcat R)$ on $\C$ are simply elements of the bifunctor 

$$\dcat{Lift} \colon  (\DBL / \Sq{\C})^\op \times (\DBL / \Sq{\C})^\op \to \SET$$

from Section~\ref{sect:lifting}.  Accordingly we define $\Cawfs{\C}$ to be the two-sided category of elements of $\dcat{Lift}$.  Its objects are lifting structures, whilst a morphism $(F_l,F_r)\colon (\dcat{L}, \phi,\dcat{R}) \to (\dcat{L'}, \phi',\dcat{R'})$ consists of morphisms $F_l\colon \dcat{L} \to \dcat{L'}$ and $F_r\colon \dcat{R'} \to \dcat{R}$ in $\DBL/\Sq{\C}$ such that 

 \begin{equation}\label{eq:oawfsmap}
\phi'_{F_{l},1} = \phi_{1,F_{r}} \in \dcat{Lift}(\dcat{L},\dcat{R'}) .
\end{equation}
This is equally to say that either of the following two diagrams are commutative
\begin{equation*}
\xymatrix{
\dcat{L} \ar[r]^-{\phi_l} \ar[d]_{F_l}  & \LLP(\dcat{R}) \ar[d]^{\LLP(F_r)} \\
\dcat{L'} \ar[r]_-{\phi'_l} & \LLP(\dcat{R'})
}
\hspace{2cm}
\xymatrix{
\dcat{R'} \ar[r]^-{\phi'_r} \ar[d]_{F_r}  & \RLP(\dcat{L'}) \ar[d]^{\RLP(F_l)} \\
\dcat{R} \ar[r]_-{\phi_r} & \RLP(\dcat{L})
}
\end{equation*}
noting that one commutes iff the other does, since they are adjoint transposes through the adjunction $\LLP \dashv \RLP$.

As a category of elements, $\Cawfs{\C}$ comes equipped with a two-sided discrete fibration

\begin{equation}\label{eq:forgetful}
\xymatrix{
& \Cawfs{\C} \ar[dl]_{\Left{}} \ar[dr]^{\Right{}} \\
\DBL / \Sq{\C} && (\DBL / \Sq{\C})^{op}
}
\end{equation}
where $\Left(\dcat L,\phi,\dcat R) = \dcat L$ and $\Right(\dcat L,\phi,\dcat R) = \dcat R$ and where both of these forgetful functors have the obvious action on morphisms.  

As with any two-sided discrete fibration, $\Left{}$ is a split fibration and $\Right{}$ is a split opfibration.  In particular, given a lifting structure $(\dcat L,\phi,\dcat R)$ and morphism $F\colon \dcat{L'} \to \dcat{L}$ in $\DBL/\Sq{C}$, its cartesian lifting along $\Left{}$ is the morphism of lifting structures $(F,1)\colon (\dcat{L'},\phi_{F,1},\dcat{R}) \to (\dcat L,\phi,\dcat R)$; likewise, given a morphism $G\colon \dcat{R} \to \dcat{R'}$ in $(\DBL/\Sq{C})^{op}$, its opcartesian lift along $\Right{}$ is the morphism $(1,G)\colon (\dcat L,\phi,\dcat R) \to (\dcat L,\phi_{1,G},\dcat R')$.

\begin{Exs}\label{ex:canonical}
Corresponding to the fact that $\dcat{Lift}$ is representable in each variable, we have that each fibre of $\Left{}$ has an initial object and each fibre of $\Right{}$ has a terminal object.  Namely,
\begin{enumerate}
\item Given $\dcat{L} \in \DBL / \Sq{\C}$, we can form the lifting structure $(\dcat{L},can,\RLP(\dcat{L}))$ when $can$ is the canonical double-categorical lifting operation having $can_{l} = \epsilon_{\dcat{L}}$ and $can_{r} = 1$.  Given a second lifting structure $(\dcat{L},\phi,\dcat{R})$ we have the unique morphism $(1,\phi_r)\colon (\dcat{L},can,\RLP(\dcat{L})) \to (\dcat{L},\phi,\dcat{R})$ with first component the identity.
\item Given $\dcat{R} \in \DBL / \Sq{\C}$, we can form the lifting structure $(\LLP({\dcat{R})},can,\dcat{R})$ when $can$ is the canonical double-categorical lifting operation having $can_{l} = 1$ and $can_{r} = \eta_{\dcat{R}}$.  Given a second lifting structure $(\dcat{L},\phi,\dcat{R})$ we have the unique morphism $(\phi_l,1)\colon  (\dcat{L},\phi,\dcat{R}) \to (\LLP({\dcat{R})},can,\dcat{R})$ with second component the identity.
\end{enumerate}
\end{Exs}

\begin{Exs}\label{exs:homotopy}

In practice,  the most common examples of lifting structures (resp. lifting awfs) are the \emph{cofibrantly generated ones}, which combine the two cases of Examples~\ref{ex:canonical}.  These are the lifting structures (resp. lifting awfs) of the form $(\LLP(\RLP({\dcat{L}))},can,\RLP(\dcat{L}))$ for $\dcat{L}$ a \emph{small} double category over $\C$.

For instance, let $\C$ be a Quillen model category and $L$ a generating set of cofibrations in $\C$.  We can view $L$ as a double category $\dcat{L}$ over $\Sq{\C}$ with one vertical morphism for each member of $L$, and only identity horizontal morphisms and squares.  Then the double category of right maps $\RLP(\dcat{L})$ consists of the \emph{algebraic trivial fibrations} \cite{Riehl2011Algebraic} --- that is, morphisms of $\C$ equipped with a choice of lifting against each generating cofibration of $L$.  

There are strong results about when such lifting structures are lifting awfs --- for instance, if $\C$ is locally presentable and $\dcat{L}$ a small double category over $\C$, then $(\LLP(\RLP({\dcat{L}))},can,\RLP(\dcat{L}))$ is always a lifting awfs.  See Proposition~\ref{prop:locallypresentable} of Section~\ref{sect:equivalence}.  In particular, if $\C$ is a \emph{combinatorial model category} --- that is, locally presentable and cofibrantly generated --- this is the case.  We refer the reader to \cite{Riehl2011Algebraic} for further homotopical examples of this kind, expressed as awfs rather than as lifting awfs (of course our main result Theorem~\ref{thm:equivalence} will establish the equivalence of the two notions.)

Let us remark that in such cofibrantly generated settings, it is quite rare that one has a complete understanding of the double category $\LLP(\RLP(\dcat{L}))$ of left maps.  One knows that the left maps include $\dcat{L}$ via the counit $\epsilon_{\dcat{L}}\colon \dcat{L} \to \LLP(\RLP(\dcat{L}))$ and that they inherit various constructions from $\C$, such as coproducts and pushout along a map, but often we only have partial information of this nature.  Obtaining a more complete understanding of $\LLP(\RLP(\dcat{L}))$ in terms of $\dcat{L}$ would be useful, with a starting point being Athorne's \emph{algebraic relative cell complexes} --- see \cite{Athorne} and \cite{AthorneThesis}.

\end{Exs}
\black

We define the categories of pre-awfs and lifting awfs on $\C$ as the full subcategories $\Pawfs{\C}$ and $\Oawfs{\C}$ of $\Cawfs{\C}$ containing the pre-awfs and lifting awfs respectively.  Accordingly, we have inclusions of full subcategories $\Oawfs{\C} \hookrightarrow \Pawfs{\C}$ and $\Pawfs{\C} \hookrightarrow \Cawfs{\C}$. We will frequently use that being a pre-awfs or lifting awfs is an \emph{isomorphism invariant property} of lifting structures.  In other words, $\Oawfs{\C}$ and $\Pawfs{\C}$ are replete full subcategories of $\Cawfs{\C}$.

\begin{Rk}\label{rk:canonical}
Observe that a lifting structure $(\dcat{L}, \phi,\dcat{R})$ is a pre-awfs if and only if both $(1,\phi_r)\colon (\dcat{L},can,\RLP(\dcat{L})) \to (\dcat{L},\phi,\dcat{R})$ and $(\phi_l,1)\colon  (\dcat{L},\phi,\dcat{R}) \to (\LLP({\dcat{R})},can,\dcat{R})$ are invertible.  Therefore, if $(\dcat{L}, \phi,\dcat{R})$ is a pre-awfs/lifting awfs so are $(\dcat{L},can,\RLP(\dcat{L}))$ and $(\LLP({\dcat{R})},can,\dcat{R})$.
\end{Rk}

The forgetful functors \eqref{eq:forgetful} restrict to forgetful functors $\Left{}$ and $\Right{}$ with domain $\Pawfs{\C}$ and, further, to forgetful functors $\Left{}$ and $\Right{}$ with domain $\Oawfs{\C}$.  In fact, when the domain is taken to be $\Pawfs{\C}$, and hence its full subcategory $\Oawfs{\C}$ too, $\Left{}$ and $\Right{}$ become fully faithful.  Before proving this claim, we prove a lemma of independent interest.

\begin{Lemma}\label{lem:fixpoints}
Let $(\dcat{L}, \phi,\dcat{R})$ be a pre-awfs.  Then the counit $\epsilon_{\dcat L}\colon \dcat{L} \to \LLP(\RLP(\dcat{L}))$ and unit $\eta_{\dcat R}\colon \dcat{R} \to \RLP(\LLP(\dcat{R}))$ of the adjunction $\LLP \dashv \RLP$ are invertible.
\end{Lemma}
\begin{proof}
As adjoint transposes of one another, the maps $\phi_l$ and $\phi_r$ are related by the commutative diagrams below.
\begin{equation*}\label{eq:triangles}
\xymatrix{
\dcat L \ar[dr]_{\phi_{l}} \ar[r]^-{\epsilon_{\dcat L}} & \LLP(\RLP(\dcat L)) \ar[d]^{\LLP(\phi_r)} \\
& \LLP(\dcat R)
}
\hspace{2cm}
\xymatrix{
\dcat R \ar[dr]_{\phi_r} \ar[r]^-{\eta_{\dcat R}} & \RLP(\LLP(\dcat R)) \ar[d]^{\RLP(\phi_l)} \\
& \RLP(\dcat L)
}
\end{equation*}
By assumption both $\phi_l$ and $\phi_r$ are invertible.  Hence so too are $\LLP(\phi_r)$ and $\RLP(\phi_l)$.  Therefore, by 2-from-3, both $\epsilon_{\dcat L}$ and $\eta_{\dcat R}$ are invertible.
\end{proof}

\begin{Prop}\label{prop:ff1}
The forgetful functors $\Left\colon \Pawfs{\C} \to \DBL/\Sq{\C}$ and \newline $\Right\colon \Pawfs{\C} \to (\DBL/\Sq{\C})^{op}$
are fully faithful.
\end{Prop}
\begin{proof}
Let $(\dcat{L}, \phi,\dcat{R})$ and $(\dcat{L'}, \phi',\dcat{R'})$ be pre-awfs, and consider a morphism
\begin{equation*}
\xymatrix{
\dcat{L} \ar[dr]_{U_L} \ar[rr]^{F_l} && \dcat{L'} \ar[dl]^{U_{L'}} \\
& \Sq{\C}
}
\end{equation*}
in $\DBL/\Sq{\C}$.  Following the proof of Lemma~\ref{lem:fixpoints}, the counit $\epsilon_L\colon \dcat{L} \to \LLP(\RLP(\dcat{L}))$ is isomorphic in $\dcat{L}/\LLP{(\thg)}$ to $\phi_l\colon \dcat{L} \to \LLP(\dcat{R})$, which is therefore also $\LLP$-couniversal.  Therefore 
there exists a unique map $F_r\colon \dcat{R'} \to \dcat{R}$ in $\DBL/\Sq{\C}$ such that the square

\begin{equation*}
\xymatrix{
\dcat{L} \ar[r]^-{\phi_l} \ar[d]_{F_l}  & \LLP(\dcat{R}) \ar[d]^{\LLP(F_r)} \\
\dcat{L'} \ar[r]_-{\phi'_l} & \LLP(\dcat{R'})
}
\end{equation*}
commutes.  This proves fully faithfulness of $\Left{}$ and the argument that $\Right{}$ is fully faithful is dual in form.
\end{proof}

\begin{Rk}
Since the adjunction $\LLP \dashv \RLP$ modifies the classical Galois connection, Richard Garner has raised the question of the extent to which the adjunction, like any Galois connection, is idempotent.  This seems unlikely in full generality but no counterexample seems to be known.  

By Lemma~\ref{lem:fixpoints} each pre-awfs gives rise to a pair of $(\LLP \dashv \RLP)$-fixpoints.  In fact idempotency is equivalent to the statement that given $U\colon \dcat{J} \to \Sq{\C}$ both $(\LLP(\dcat{J}),can,\RLP(\LLP(\dcat{J})))$ and $(\LLP(\RLP({\dcat{J}))},can,\RLP(\dcat{J}))$ are pre-awfs.

Certainly there exist lifting structures that are not lifting awfs yet give rise to pairs of $(\LLP \dashv \RLP)$-fixpoints.  To see this, it suffices by the above, to exhibit a pre-awfs which is not a lifting awfs.  To this end, consider a pre-factorisation system $(\E,\M)$ on $\C$ which is not a factorisation system.  An example of such is discussed in Section 3 of \cite{Cassidy}.  Following the argument of Examples~\ref{ex:second}.(\ref{item:prefact}) unchanged, the triple $(\dcat{D}(\E),\phi,\dcat{D}(\M))$ is a pre-awfs but does not satisfy the axiom of factorisation.
\end{Rk}

\section{Background on algebraic weak factorisation systems}\label{sect:awfs}

In this short section we recall the concept of an algebraic weak factorisation system (awfs) and the necessary background about them.  Nothing here is new and our treatment closely follows \cite{Bourke2016Accessible}, to which we refer the reader for further details.

\subsection{Algebraic weak factorisation systems and their morphisms}\label{sec:algebr-weak-fact}
An awfs on $\C$ begins with a
\emph{functorial factorisation}\begin{footnote}{The notion of a functorial factorisation appears to have first appeared in \cite{Linton}.}\end{footnote}: a functor $\C^{\atwo} \to \C^{\mathbf
  3}$ from the category of arrows to that of composable pairs which is
a section for the composition map $\C^{\mathbf 3} \to \C^{\atwo}$. The
action of this functor at an object $f$ or morphism $(h,k) \colon  f
\to g$ of $\C^{\atwo}$ is depicted as on the left or right in:
\[
f = X \xrightarrow{\l f} Ef \xrightarrow{\r f} Y
\qquad \quad \qquad
\cd{
 X
  \ar[r]^{\l f} \ar[d]_{h} & Ef \ar[d]|{E(h, k)} \ar[r]^{\r f} & Y
  \ar[d]^{k} \\
 W \ar[r]^{\l g} & Eg \ar[r]^{\r g} & Z\rlap{ .}}
\]
From these data we obtain endofunctors $L, R \colon  \C^{\atwo} \to
\C^\atwo$, together with natural transformations $\epsilon \colon  L
\Rightarrow 1$ and $\eta \colon  1 \Rightarrow R$ with respective $f$-components:
\begin{equation}
\cd{
A \ar[r]^1 \ar[d]_{\l f} & A \ar[d]^{f} \\
Ef \ar[r]^{\r f} & B
} \qquad \text{and} \qquad
\cd{
A \ar[r]^{\l f} \ar[d]_{f} & Ef \ar[d]^{\r f} \\
B \ar[r]^{1} & B\rlap{ .}
}\label{eq:4}
\end{equation}

An \emph{awfs} $(\mathsf L, \mathsf R)$
on $\C$ is a functorial factorisation as above, together with natural
transformations $\Delta \colon  L \to LL$ and $\mu \colon  RR \to R$
making $\mathsf L = (L,\epsilon, \Delta)$ and $\mathsf R =
(R,\eta,\mu)$ into a comonad and a monad respectively. The monad and
comonad axioms, together with the form~\eqref{eq:4} of $\eta$ and
$\epsilon$, force the components of $\Delta$ and $\mu$ at $f$ to be as
on the left and right in
\[
 \cd{
A \ar[d]_{\l f} \ar[r]^1 & A \ar[d]^{\l{\l f}} \\
Ef \ar[r]^{\c f} & E\l f}
\qquad 
\cd{
Ef \ar[d]_{\l{\r f}} \ar[r]^{\c f} & E\l f
\ar[d]^{\r{\l f}} \\
E\r f \ar[r]^{\m f} & Ef
}
\qquad 
\cd{
E \r f \ar[r]^{\m f} \ar[d]_{\r{\r f}} & Ef \ar[d]^{\r f} \\
B \ar[r]^1 & B\rlap{ ,}}
\]
and imply moreover that the middle square is the component at $f$ of a
natural transformation $\delta \colon  LR \Rightarrow RL$. The final
axiom for an \awfs is that this $\delta$ should constitute a
\emph{distributive law} of
$\mathsf L$ over $\mathsf R$.

A morphism between \awfs $(\mathsf L, \mathsf R)$ and $(\mathsf L', \mathsf R')$ on $\C$ is given by a natural family
of maps $K_f$ rendering commutative:
$$  \cd[@-1em@C-0.5em]{
    & A \ar[dl]_{\l {f}} \ar[dr]^{\lp f} \\
    Ef \ar[rr]^{K_f} \ar[dr]_{\r {f}} & & 
    E'f\rlap{ } \ar[dl]^{\rp f} \\ & B} 
$$
and such that the induced $(1,K) \colon  L \to L'$ and $(K,1) \colon  R \to R'$ are respectively a monad morphism and a comonad morphism.  These form the morphisms of the category $\Awfs{\C}$ of \awfs on $\C$.

\subsection{Double-categorical semantics}\label{sect:double}
Given an \awfs $( \mathsf L, \mathsf R)$ on $\C$ we can consider the
Eilenberg--Moore categories $\Coalg{L}$ and $\Alg{R}$ of coalgebras
and algebras over $\C^{\atwo}$; these are thought of as providing the
respective left and right classes of the awfs. Because $\mathsf L$ is a comonad over the
 domain functor, a coalgebra structure $f \to Lf$ on $f \colon  A \to B$ necessarily has its domain
component an identity, and so is determined by a single map
$s \colon  B \to Ef$; we write such a coalgebra as
$\alg f = (f, s) \colon  A \to B$. Dually, an $\mathsf R$-algebra
structure on $g \colon  C \to D$ is determined by a single map
$p \colon  Ef \to C$, and will be denoted
$\alg g = (g,p) \colon  C \to D$.

An $\mathsf R$-algebra morphism
\begin{equation*}
\cd{
  A \ar[d]_{\alg f} \ar[r]^u & C \ar[d]^{\alg g} \\
  B \ar[r]_{v} & D}\label{eq:5}
\end{equation*}
is a commuting square in $\C$ compatible with the algebra structures
on $\alg f$ and $\alg g$; similar notation and conventions will be
used for $\mathsf L$-coalgebras. 

The $\mathsf R$-algebras admit a composition law, satisfying the compatibilities for a concrete double category $\DAlg{R}$ over $\C$.  The forgetful double functor 
$V^\mathsf R \colon  \DAlg{R} \to \Sq{\C}$, displayed as an internal functor
between internal categories in $\Cat$, is as below.
\begin{equation*}
\xymatrix{
  \Alg{R} \ar@<-5pt>[d]_d\ar@<5pt>[d]^c \ar[r]^{V^\mathsf R} &
  \C^\atwo
  \ar@<-5pt>[d]_d\ar@<5pt>[d]^c \\
  \ar[u]|i \C \ar[r]_1 & \C \ar[u]|i } 
  \end{equation*}

Similarly, the $\mathsf L$-coalgebras constitute a concrete double category $U^{\mathsf L} \colon  \DCoalg{L} \to \Sq{\C}$.

Each morphism $K \colon  (\mathsf L, \mathsf R) \to (\mathsf L', \mathsf R')$ of \awfs on $\C$ induces morphisms of concrete double
categories
\begin{equation*}
\cd{
\DCoalg{L} \ar[dr]_{U^\mathsf L} \ar@{.>}[rr]^{\DCoalg{K}} && \DCoalg{L'}
\ar[dl]^{U^\mathsf{L'}} \\
& \Sq{\C} 
}
\hspace{1cm}
\cd{
\DAlg{R'} \ar[dr]_{U^\mathsf R} \ar@{.>}[rr]^{\DAlg{K}} && \DAlg{R}
\ar[dl]^{U^\mathsf{R'}} \\
& \Sq{\C} 
}
\end{equation*}
and in this way we obtain functors $$\DCoalg{(\thg)} \colon  \Awfs{\C} \to \DBL/\Sq{\C} \hspace{0.5cm} \textnormal{and} \hspace{0.5cm}  \DAlg{(\thg)} \colon  \Awfs{\C}^{op} \to \DBL/\Sq{\C},$$ 
both of which are in fact fully faithful by Lemma 6.9 of \cite{Riehl2011Algebraic}.\begin{footnote}{Lemma 6.9 of \cite{Riehl2011Algebraic} actually proves a stronger result, allowing the category $\C$ to vary.  The stated case of fixed $\C$ trivially follows from this stronger version.}\end{footnote}

The Beck theorem for awfs characterizes those concrete double categories in the essential image of $\DCoalg{(\thg)}$ and $\DAlg{(\thg)}$.  Focusing on the latter, it says that $\DAlg{(\thg)}$ has in its essential image precisely those concrete double categories $U\colon \dcat{A} \to \Sq{\C}$ which are monadic and right-connected.  Here, monadic means that $U_1\colon \A_1 \to \C^{\atwo}$ is strictly monadic.
 Right connected means that the codomain functor $c\colon \A_1 \to \A_0$ be left adjoint to the identities functor $i\colon \A_0 \to \A_1$ with identity counit. (See Section 3.5 of \cite{Bourke2016Accessible} for this formulation of right connectedness).  When $\dcat{A}$ is a concrete double category over $\C$, this simply amounts to the statement that for each vertical morphism $\alg f\colon A \to B$ of $\dcat{A}$ the commutative square $(f,1_B)\colon f \to 1_B$ lifts to a square $\alg f \to \alg 1_B=i(B)$ in $\A$.

\section{The equivalence between awfs and lifting awfs}\label{sect:equivalence}

In this section we construct the semantics functor $$\Sem\colon \Awfs{\C} \to \Oawfs{\C}$$ before proving, in Theorem~\ref{thm:equivalence}, that it is an equivalence of categories.  Using this, we re-interpret the Beck theorem for awfs in terms of lifting awfs and establish a slight redundancy in the axioms for a lifting awfs.

\subsection{The semantics functor}

Let us start by observing that each awfs $(\mathsf L, \mathsf R)$ gives rise to a lifting awfs $(\DCoalg{L},\Phi,\DAlg{R})$.  Here, all of the ingredients are well known.  By Proposition 20 of \cite{Bourke2016Accessible}, there is a canonical $(\DCoalg{L},\DAlg{R})$-lifting operation $\Phi$ satisfying the axiom of lifting.  At an $\mathsf L$-coalgebra $\alg f = (f,s)$, an $\mathsf
R$-algebra $\alg g = (g,p)$ and a commuting square
\begin{equation*}
\cd{
A \ar[d]_{f} \ar[rr]^{u}  && C \ar[d]^{g} \\
B \ar@{.>}[urr]|{\Phi_{\alg f, \alg g}(u,v)} \ar[rr]_{v}  && D} \qquad \qquad
\end{equation*}
the diagonal filler is given by the composite $p \circ E(u,v) \circ s \colon  B \to Ef \to Eg \to C$.  In the factorisation $f=Rf \circ Lf$, the morphisms $(Lf,1)\colon f \to Rf$ and $(1,Rf)\colon Lf \to f$ are the unit and counit of the monad $\mathsf R$ and comonad $\mathsf L$ --- the universal properties of the unit and counit verify the axiom of factorisation.

Each morphism $K\colon (\mathsf L, \mathsf R) \to (\mathsf L', \mathsf R')$ of awfs on $\C$ gives rise to a morphism $$(\DCoalg{K},\DAlg{K})\colon  (\DCoalg{L},\Phi,\DAlg{R}) \to (\DCoalg{L'},\Phi',\DAlg{R'})$$ of lifting awfs.  Indeed, the defining equation ~\eqref{eq:oawfsmap} says that given an $\mathsf L$-coalgebra $(f,r)$, an $\mathsf R'$-algebra $(g,s)$ and a commutative square
\begin{equation*}
\cd{
A \ar[d]_{f} \ar[rr]^{u}  && C \ar[d]^{g} \\
B \ar[rr]_{v}  && D} 
\end{equation*}
the following diagram commutes
\begin{equation*}
\xymatrix{
B \ar[r]^{r} \ar[dr]_{Kf\circ r} & Ef \ar[d]^{Kf} \ar[rr]^{E(u,v)} && Eg \ar[d]_{Kg} \ar[r]^{s\circ Kg} & D \\
& E'f \ar[rr]_{E'(u,v)} && E'g \ar[ur]_{s}
}
\end{equation*}
and this holds by naturality of $K$.  It follows that we obtain a functor $\Sem\colon \Awfs{\C} \to \Oawfs{\C}$ sending $(\mathsf L, \mathsf R)$ to $(\DCoalg{L},\Phi,\DAlg{R})$ and with the above action on morphisms.  All told, we obtain a commutative diagram of categories and functors as below.
\begin{equation}\label{eq:commutative}
\xymatrix{
&& \DBL/\Sq{\C} \\
\Awfs{\C} \ar[drr]_{\DAlg{(\thg)}} \ar[urr]^{\DCoalg{(\thg)}} \ar[rr]^{\Sem} && \Oawfs{\C} \ar[u]_{\Left{}} \ar[d]^{\Right{}} \\
&& (\DBL/\Sq{\C})^{op}
}
\end{equation}

\begin{Prop}\label{prop:ff}
Each of the functors in Diagram~\ref{eq:commutative} is fully faithful.
\end{Prop}
\begin{proof}
By Proposition~\ref{prop:ff1}, the forgetful functors $\Left\colon \Pawfs{\C} \to \DBL/\Sq{\C}$ and $\Right\colon \Pawfs{\C} \to (\DBL/\Sq{\C})^{op}$ are fully faithful.  Precomposing these by the full inclusion $\Oawfs{\C} \hookrightarrow \Pawfs{\C}$ ensures that the two vertical morphisms in the diagram are fully faithful.  As recalled in Section~\ref{sect:double}, the two diagonals are fully faithful by Lemma 6.9 of \cite{Riehl2011Algebraic}. Therefore $\Sem{}$ is fully faithful by $2$-from-$3$ applied to either triangle.
\end{proof}

The main result of the paper, Theorem~\ref{thm:equivalence} below, shows that $\Sem\colon \Awfs{\C} \to \Oawfs{\C}$ is an equivalence of categories.  Having established fully faithfulness, it remains to prove essential surjectivity.  Since $\Right{} \colon  \Oawfs{\C} \to (\DBL/\Sq{\C})^{op}$ is fully faithful, essentially surjectivity of $\Sem$ amounts to showing that for each lifting awfs $(\dcat{L}, \phi,\dcat{R})$, there exists an awfs $(\mathsf L, \mathsf R)$ and an isomorphism $\dcat{R} \cong \DAlg{R}$ in $\DBL/\Sq{\C}$ or, equivalently, an isomorphism $\RLP{(\dcat L)} \cong \DAlg{R}$ in $\DBL/\Sq{\C}$.  To establish this will require us to study concrete double categories of the form $V\colon \RLP{(\dcat L)} \to \Sq{\C}$ more closely, and then apply the Beck theorem for awfs.  The key technical result is the following.

\begin{Lemma}\label{lem:Beck}
Consider a double category $U\colon \dcat{L} \to \Sq{\C}$ over $\Sq{\C}$.  Then $V_1\colon \RLP{(\dcat{L})}_1 \to \C^{\atwo}$ creates any colimits that are preserved by $\C^{\atwo}(Uj,-)\colon \C^{\atwo} \to \Set$ for each vertical arrow $j \in \dcat{L}$.  In particular, $V_1$ creates $V_1$-absolute colimits.\begin{footnote}{Here we refer to the strict form of creation --- given  a functor $U\colon \A \to \B$, diagram $D\colon \J \to \A$ and colimit cocone $p\colon UD \to \Delta{x}$, we say that $U$ creates this colimit if, firstly, there exists a unique cocone $p'\colon D \to \Delta(x')$ such that $Up'=p$ and, secondly, the lifted cocone is itself a colimit cocone.}\end{footnote}
\end{Lemma}
\begin{proof}
We begin by reformulating what it means to give an object or morphism of $\RLP{(\dcat{L})}_1$ in more convenient terms.  Firstly, observe that given $a\colon A_0 \to A_1$ and $b\colon B_0 \to B_1$ we can form the induced map $$\delta_{a,b}\colon \C(A_1,B_0) \to \C^{\atwo}(a,b)\colon f \mapsto (f\circ a,b\circ f)$$ and to give a section $\theta_{a,b}\colon \C^{\atwo}(a,b) \to \C(A_1,B_0)$ of $\delta_{a,b}$ is precisely to give a diagonal filler in each square on the left below.
\begin{equation*}
\cd{A_0 \ar[d]_{a} \ar[r]^-{r_0} &  B_0 \ar[d]^{b}  \\
A_1 \ar@{.>}[ur]|{\theta_{a,b}(r_0,r_1)} \ar[r]^-{r_1}  &B_1}
\hspace{1cm}
\cd{UJ_0 \ar[d]_{Uj} \ar[r]^-{r_0} &  F_0 \ar[d]^{f}  \\
UJ_1 \ar@{.>}[ur]|{\theta_{j}(r_0,r_1)} \ar[r]^-{r_1}  &F_1}
\hspace{1cm}
\cd{\C^{\atwo}(Uj,f)\ar[d]_{\C^{\atwo}(Uj,r)} \ar[r]^-{\theta_j} &  \C(UJ_1,F_0) \ar[d]^{\C(UJ_1,r_0)}  \\
\C^{\atwo}(Uj,g) \ar[r]^-{\phi_{j}}  &\C(UJ_1,G_0)}
\end{equation*}
Building on this point of view, an $(\dcat{L},f)$-lifting operation as in the central diagram above is specified by sections $\theta_{j}\colon \C^{\atwo}(Uj,f) \to \C(UJ_1,F_0)$ of $\delta_{Uj,f}$ for each vertical arrow $j\colon J_0 \to J_1$ of $\dcat{L}$, subject to vertical and horizontal compatibility of the liftings.  Moreover, given $(f,\theta), (g,\phi) \in \RLP{(\dcat{L})}_1$ a square $r=(r_0,r_1)\colon f \to g$ is a morphism of $\RLP{(\dcat{L})}_1$ just when the square above right
is commutative for each vertical arrow $j \in \dcat{L}$.

With this reformulation in hand, let us consider a diagram $(d,\phi)\colon I \to \RLP{(\dcat{L})}_1$.  Such is specified by morphisms
\begin{equation*}
\cd{ D^a_0 \ar[d]_{(d^a,\phi^a)} \ar[r]^-{D^{\alpha}_0} &  D^b_0 \ar[d]^{(d^b,\phi^b)}  \\
D^a_1 \ar[r]^-{D^{\alpha}_1}  &D^b_1}
\end{equation*}
of $\RLP{(\dcat{L})}_1$ for each $\alpha\colon a \to b \in I$.   Suppose that the colimit $p\colon d \to \Delta(c)$ of the underlying diagram $d\colon I \to \C^{\atwo}$ exists and is preserved by each representable of the form $\C^{\atwo}(Uj,-)$.  Firstly, we must show that $c$ admits a unique structure $(c,\theta) \in \RLP{(\dcat{L})}_1$ such that each $p^a\colon d^a \to c$ preserves the liftings.

Now since $\C^{\atwo}(Uj,-)$ preserves the colimit, the left vertical morphism in the diagram below forms the component of a colimit cocone. 

\begin{equation*}
\cd{\C^{\atwo}(Uj,d^a) \ar[d]_{\C^{\atwo}(Uj,p^a)} \ar[r]^-{\phi^a_j} &\C(UJ_1,D^a_0)  \ar[d]|{\C(UJ_1,p^a_0)}  \ar[r]^{\delta_{Uj,d^a}} & \C^{\atwo}(Uj,d^a) \ar[d]^{\C^{\atwo}(Uj,p^a)}\\
 \C^{\atwo}(Uj,c) \ar@{.>}[r]^-{\exists ! \theta_{j}}  & \C(UJ_1,C_0) \ar[r]^{\delta_{Uj,c}} & \C^{\atwo}(Uj,c)
 }
\end{equation*}

Therefore, by its universal property, there exists a unique $\theta_j$ making the left square commute.  It remains to show that $(c,\theta) \in \RLP{(\dcat{L})}_1$.

To see that $\theta_j$ is a section of $\delta_{Uj,c}$ - i.e. that the composite on the bottom row is an identity --- we again use the fact that the left vertical is the component of a colimit cocone, and that the composite on the top row is an identity.

For the horizontal compatibility of liftings, consider $r\colon j \to k \in \dcat{L}_1$.  We must show that the front face of the cube below commutes.  
\begin{equation*}
\begin{tikzcd}[row sep=2.5em]
\C^{\atwo}(Uj,d^a) \arrow[rr,"\phi^a_j"] \arrow[dr,swap,"(p^a)_*"] \arrow[dd,swap," (Ur)^*"] &&
   \C(UJ_1,D^a_0) \arrow[dd,swap,"(Ur_1)^*" near start] \arrow[dr,"(p^a_0)_*"] \\
&  \C^{\atwo}(Uj,c) \arrow[rr,crossing over,"\theta_j" near start] &&
  \C(UJ_1,C_0)\arrow[dd,"(Ur_1)^*"] \\
\C^{\atwo}(Uk,d^a) \arrow[rr,"\phi^a_k" near end] \arrow[dr,swap,"(p^a)_*"] &&   \C(UK_1,D^a_0) \arrow[dr,swap,"(p^a_0)_*"] \\
& \C^{\atwo}(Uk,c) \arrow[rr,"\theta_k'"] \arrow[uu,<-,crossing over,"(Ur)^*" near end]&& \C(UK_1,C_0)
\end{tikzcd}
\end{equation*}
The back face commutes by assumption.  The two faces on the $(y,z)$-axes commute trivially whilst the two faces on the $(x,z)$-axes commute by construction of $\theta_j$ and $\theta_k$, so that all faces but the front one are known to commute.  It follows by diagram chasing that all paths from top left to bottom right commute --- in particular, the two paths of the front face commute on precomposition with the colimit coprojection $(p^a)_*\colon \C^{\atwo}(Uj,d^a) \to \C^{\atwo}(Uj,c)$.  Therefore, by the universal property of the colimit, the front face commutes too.  

Next, we establish compatibility of the liftings $\theta$ with respect to a composite $k \circ j$ of vertical morphisms in $\dcat{L}$.  To this end, consider a commutative square $(r_0,r_1)\colon Uk  \circ Uj \to c$ and define $\theta_{k,j}(r_0,r_1)$ to be the diagonal $\theta_k(\theta_j(r_0,r_1 \circ Uk),r_1)$ obtained by first lifting against $Uj$ and then against $Uk$.  Similarly, given $(r_0,r_1)\colon Uk \circ Uj \to d^a$, we write $\theta^a_{k,j}(r_0,r_1)$ for the corresponding diagonal, obtained by first lifting against $j$ and then $k$.  Since the colimit coprojections $p^a\colon d^a \to c$ preserve liftings by construction, they also preserve liftings constructed in two steps as above, so that the following diagram is serially commutative.

\begin{equation*}
\xymatrix{
\C^{\atwo}(UkUj,d^a) \ar@<1ex>[rr]^{\theta^a_{k,j}} \ar@<-1ex>[rr]_{\theta^a_{k\circ j}} \ar[d]_{(p^a)_*} && \C^{\atwo}(UK_1,D^a_0) \ar[d]^{(p^a_0)_*} \\
\C^{\atwo}(UkUj,c)\ar@<1ex>[rr]^{\phi_{k,j}} \ar@<-1ex>[rr]_{\phi_{k \circ j}} && \C^{\atwo}(UK_1,C_0)
}
\end{equation*}
As each $(D^a,\theta^a)$ is an object of $\RLP{(\dcat{L})}_1$, the two arrows on the top row coincide.  Therefore, using the universal property of the colimit coprojections in the left vertical, the two arrows on the bottom row coincide, as required.  

This concludes the proof that $p\colon d \to \Delta(c)$  lifts uniquely to a cocone $p\colon (d,\phi) \to (c,\theta)$, and it remains to show that it exhibits $(c,\theta)$ as the colimit.  Given a second cocone $q\colon (d,\phi) \to (f,\rho)$, there exists a unique ${q'}\colon c \to f$ such that $q' \circ p^a = q^a$ for each $a$.  To show that $q'\colon c \to f$ is a morphism in $\RLP{(\dcat{L})}_1$, we must show that the lower square in the diagram below commutes.

\begin{equation*}
\cd{\C^{\atwo}(Uj,d^a) \ar@/_3.5pc/[dd]_{(q^a)_*} \ar[d]^{(p^a)_*} \ar[r]^-{\phi^a_j} &\C(UJ_1,D^a_0)  \ar[d]_{(p^a_0)_*} \ar@/^3.5pc/[dd]^{(q^a)_*}\\
 \C^{\atwo}(Uj,c)\ar[d]^{(q')_*} \ar@{.>}[r]^-{\theta_{j}} &  \C^{\atwo}(UJ_1,C_0) \ar[d]_{(q'_0)_*} \\
  \C^{\atwo}(Uj,f) \ar@{.>}[r]^-{\rho_j} & \C(UJ_1,F_0)  
 }
\end{equation*}
Since the upper square and outer diagram are commutative, this follows on using that the upper left vertical map is a colimit coprojection.

\end{proof}

\begin{Prop}\label{prop:Beck}
Consider a double category $\dcat{L}$ over $\Sq{\C}$.  There exists an awfs $(\mathsf L, \mathsf R)$ and isomorphism $\DAlg{R} \cong \RLP{(\dcat{L})}$ in $\DBL/\Sq{\C}$ if and only if $V_1\colon \RLP{(\dcat{L})}_1 \to \C^{\atwo}$ has a left adjoint.
\end{Prop}
\begin{proof}
By the Beck theorem for awfs, recalled in Section~\ref{sect:double} above, there is an isomorphism $\DAlg{R} \cong \RLP{(\dcat{L})}$ in $\DBL/\Sq{\C}$ if and only if $V\colon \RLP{(\dcat{L})} \to \Sq{\C}$
 is a concrete double category over $\Sq{\C}$ which is right connected and has $V_1\colon \RLP(\dcat{L}) \to \C^{\atwo}$ strictly monadic.  
 
Now $V\colon \RLP{(\dcat{L})} \to \Sq{\C}$ is a concrete double category by construction.  For right connectedness, we must show that for each vertical morphism $(f,\phi)\colon A \to B \in \RLP(\dcat{L})$ the square below left
\begin{equation*}
\xymatrix{
A \ar[r]^-{f} \ar[d]_{(f,\phi)}  & B \ar[d]^{(1_B,!)} \\
B\ar[r]_-{1} & B
}
\hspace{1cm}
\xymatrix{
C \ar[r]^-{r} \ar[d]_{g}  & B \ar[d]^{1_B} \\
D \ar@{.>}[ur]^{s} \ar[r]_-{s} & B
}
\end{equation*}
is a square in $\RLP(\dcat{L})$.  Note that here $1_B\colon B \to B$ is equipped with the unique possible lifting function since each square with target $1_B$ has a unique diagonal filler, as depicted above right.  But since the condition for $(f,1_B)\colon (f,\phi) \to (1_B,!)$ to be a square, is about the equality of two diagonal fillers for a square with target $1_B$, this condition follows from the uniqueness of such diagonal fillers.

Moreover, by Beck's monadicity theorem, $V_1$ will be strictly monadic if and only if it has a left adjoint and creates $V_1$-absolute coequalisers.  By Lemma~\ref{lem:Beck}, $V_1$ always creates $V_1$-absolute coequalisers.  Therefore it is strictly monadic if and only it has a left adjoint, completing the proof.
\end{proof}

\begin{Thm}\label{thm:equivalence}
The semantics functor $\Sem\colon \Awfs{\C} \to \Oawfs{\C}$ is an equivalence of categories.
\end{Thm}
\begin{proof}
By Proposition~\ref{prop:ff}, $\Sem$ is fully faithful.  Since, by the same result, $\Right{}$ is fully faithful, essential surjectivity of $\Sem$ amounts to showing that given a lifting awfs $(\dcat{L},\phi,\dcat{R})$ there exists an awfs $(\mathsf L,\mathsf R)$ such that $\dcat{R} \cong \DAlg{R}$ in $\DBL/\Sq{\C}$.  But since $\dcat{R} \cong \RLP(\dcat{L})$ in $\DBL/\Sq{\C}$, this is equally to say that $\RLP(\dcat{L}) \cong \DAlg{R}$ in $\DBL/\Sq{\C}$.  By Proposition~\ref{prop:Beck}, this will be the case if and only if $U_1\colon \RLP(\dcat{L})_1 \to \C^{\atwo}$ has a left adjoint, but it does so since by the definiton of lifting awfs, the isomorphic $U_1\colon \R_1 \to \C^{\atwo}$ has a left adjoint.
\end{proof}

In light of the above theorem, let us revisit the diagram ~\eqref{eq:commutative}, repeated below for the reader's convenience.

\begin{equation*}
\xymatrix{
&& \DBL/\Sq{\C} \\
\Awfs{\C} \ar[drr]_{\DAlg{(\thg)}} \ar[urr]^{\DCoalg{(\thg)}} \ar[rr]^{\Sem} && \Oawfs{\C} \ar[u]_{\Left{}} \ar[d]^{\Right{}} \\
&& (\DBL/\Sq{\C})^{op}
}
\end{equation*}

Since the horizontal functor is an equivalence, we can reinterpret Beck's theorem for awfs to characterise the essential images of the vertical maps.  

\begin{Prop}\label{prop:oBeck}
$U\colon \dcat{R} \to \Sq{\C}$ lies in the essential image of $\Right\colon \Oawfs{\C} \to (\DBL/\Sq{\C})^{op}$ if and only if
\begin{enumerate}
\item $U_0$ is invertible;
\item $U_1$ is strictly monadic;
\item $\dcat{R}$ is right connected.
\end{enumerate}
Indeed, if these conditions are satisfied, then $(\LLP(\dcat{R}),can,\dcat{R})$ is a lifting awfs on $\C$ --- moreover, it is the essentially unique lifting awfs with right part $\dcat{R}$.
\end{Prop}
\begin{proof}
Since $\Sem$ is an equivalence, both $\DAlg{(\thg)}$ and $\Right{}$ have the same essential image.  It follows immediately from the Beck theorem for awfs that $\DAlg{(\thg)}$ factors as an equivalence through the full subcategory of $(\DBL/\Sq{\C})^{op}$ containing the right-connected monadic concrete double categories over $\C$.  These do not, however, form a replete full subcategory because the condition for a concrete double category that $U_0$ is the identity is not isomorphism invariant --- rather the repletion, and hence essential image, consists precisely of those $U\colon \dcat{R} \to \Sq{\C}$ satisfying the stated hypotheses.  

Having characterised the essential image, it remains to verify the final claim. By the first part of the result there exists a lifting awfs $(\dcat{L},\phi,\dcat{R'})$ and isomorphism $G\colon \dcat{R'} \cong \dcat{R}$ in $(\DBL/\Sq{\C})^{op}$.  This lifts along the opfibration $\Right\colon \Cawfs{\C} \to (\DBL/\Sq{\C})^{op}$ to the isomorphism $(1,G)\colon (\dcat{L},\phi,\dcat{R'}) \cong (\dcat{L},\phi_{1,G},\dcat{R})$ of lifting structures.  Since the lhs of the isomorphism is a lifting awfs, the rhs $(\dcat{L},\phi_{1,G},\dcat{R})$ is a lifting awfs too.  Since it is, in particular, a pre-awfs, Remark~\ref{rk:canonical} ensures that the map $(\phi'_l,1)\colon  (\dcat{L},\phi',\dcat{R}) \to (\LLP({\dcat{R})},can,\dcat{R})$ is an isomorphism of lifting structures; hence $(\LLP({\dcat{R})},can,\dcat{R})$ is a lifting awfs too.

The final claim about essential uniqueness follows from the fully faithfulness of $\Right{}$, established in Proposition~\ref{prop:ff}.
\end{proof}

The dual conditions characterise the essential image of $\Left\colon \Oawfs{\C} \to \DBL/\Sq{\C}$.  We also give the lifting awfs variant of Proposition~\ref{prop:Beck} and again there is a corresponding dual version.

\begin{Prop}\label{prop:oadjoint}
Consider a double category $\dcat{L}$ over $\Sq{\C}$.  Then the lifting structure $(\LLP(\RLP(\dcat{L})),can,\RLP(\dcat{L}))$ is a lifting awfs if and only if $V_1\colon \RLP{(\dcat{L})}_1 \to \C^{\atwo}$ has a left adjoint.
\end{Prop}
\begin{proof}
As in the proof of Proposition~\ref{prop:Beck}, if $V_1$ has a left adjoint then $V\colon \RLP{(\dcat{L})} \to \Sq{\C}$ satisfies the conditions of Proposition~\ref{prop:oBeck}, proving the result.
\end{proof}

Proposition~\ref{prop:oadjoint} implies the slight redundancy in the axiom of factorisation.

\begin{Prop}\label{prop:redundant}
A lifting awfs is a pre-awfs such that each morphism $f$ in $\C$ admits a factorisation $f = U_1 g \circ V_1 h$ such that either $(1_A,V_1 f)\colon U_1 g \to f$ is $U_1$-couniversal \emph{or} $(U_{1}g,1_B)\colon f \to V_{1}h$ is $V_1$-universal.
\end{Prop}
\begin{proof}
It suffices to prove the case that the factorisations are $V_1$-universal --- the other case is dual.  So let $(\dcat{L},\phi,\dcat{R})$ be a pre-awfs admitting factorisations $f = U_1g \circ V_1 h$ such that $(U_{1}g,1_B)\colon f \to V_{1}h$ is $V_1$-universal.  As a pre-awfs, it follows by Remark~\ref{rk:canonical} that the induced map $(1,\phi_r)\colon (\dcat{L},can,\RLP(\dcat{L})) \to (\dcat{L},\phi,\dcat{R})$ is invertible, so the lhs is also a pre-awfs.  By Remark~\ref{rk:canonical} again, $(\epsilon_{\dcat{L}},1)\colon (\dcat{L},can,\RLP(\dcat{L})) \to (\LLP(\RLP(\dcat{L})),can,\RLP(\dcat{L}))$ is also invertible.  Therefore it suffices to prove that $(\LLP(\RLP(\dcat{L})),can,\RLP(\dcat{L}))$ is a lifting awfs.  Now $V_1\colon \R_1 \to \C^{\atwo}$ has a left adjoint by assumption, hence so too does the isomorphic $V_1\colon \RLP(\dcat{L})_1 \to \C^{\atwo}$ so that $(\LLP(\RLP(\dcat{L})),can,\RLP(\dcat{L}))$ is a lifting awfs by Proposition~\ref{prop:oadjoint}.
\end{proof}

 Let us also mention the following result, which we quoted in Examples~\ref{exs:homotopy}.

\begin{Prop}\label{prop:locallypresentable}
Suppose that $\C$ is locally presentable and $\dcat{L}$ is a small double category over $\Sq{\C}$.  Then the lifting structure $(\LLP(\RLP(\dcat{L})),can,\RLP(\dcat{L}))$ is a lifting awfs.
\end{Prop}
\begin{proof}
This is the lifting awfs version of Proposition 23 of \cite{Bourke2016Accessible}, translated across the equivalence of Theorem~\ref{thm:equivalence}.  Let us summarise the proof.  To prove the claim is, by Proposition~\ref{prop:oadjoint}, to show that $V_1:\RLP{(\dcat{L})}_1 \to \C^{\atwo}$ has a left adjoint.  For this, one uses the local presentability of $\C$ exactly as in the proof of Proposition 23 of \cite{Bourke2016Accessible}.
\end{proof}\black

Until now, we have focused upon the relationship between awfs and lifting awfs over a fixed base $\C$.  We conclude by mentioning the trivial adaptations to a varying base.  Following the terminology of \cite{Bourke2016Accessible}, let $\DBL$ be the 2-category of double categories, double functors and horizontal transformations.  We define an oplax morphism $(\dcat{L},\phi,\dcat{R}) \to (\dcat{L'},\phi',\dcat{R'})$ of lifting awfs simply to be a morphism between the left parts $\dcat{L} \to \Sq{\C}$ and $\dcat{L'} \to \Sq{\D}$ in $\DBL^{\atwo}$.  Defining $2$-cells in the same way, we obtain a $2$-category $\LOplax$ and a fully faithful $2$-functor $\Left\colon \LOplax \to \DBL^{\atwo}$.  As proven in Lemma 6.9 of \cite{Riehl2011Algebraic} (also Proposition 2 of \cite{Bourke2016Accessible}) there is a fully faithful semantics $2$-functor $\DCoalg{(\thg)} \colon  \Oplax \to \DBL^\atwo$.  This factors through $\Left{}$ as a fully faithful $2$-functor $\Sem_{\textnormal{oplax}}\colon \Oplax \to \LOplax$ which, by Theorem~\ref{thm:equivalence}, is essentialy surjective, and so a $2$-equivalence.  A similar result holds for lax morphisms.

\end{document}